\documentclass{article}
\DeclareUnicodeCharacter{0301}{\'{e}}
\usepackage[english]{babel}
\usepackage[letterpaper,top=2cm,bottom=2cm,left=3cm,right=2.5cm,marginparwidth=1.75cm]{geometry}
\usepackage[utf8]{inputenc}
\usepackage{dsfont}
\usepackage{amsmath}
\usepackage{graphicx}
\usepackage[numbers]{natbib}
\usepackage{amssymb}
\usepackage{latexsym}
\usepackage{amsmath}
\usepackage{bbm}
\usepackage{amsthm}
\usepackage{psfrag}
\usepackage{helvet}
\usepackage{enumerate}
\usepackage{fancyhdr}
\usepackage{ifthen}
\usepackage{mathtools}
\usepackage{tikz}
\usepackage{authblk}

\newcommand{\e}{\varepsilon}    
 
\newcommand{\R}{\mathbb{R}} 
\newcommand{\La}{\mathcal{L}}
\newcommand{\He}{\mathcal{H}}
\usepackage[colorlinks=true, allcolors=blue]{hyperref}

\newtheorem{ter}{{\it \textbf{Theorem}}}[section]
\newtheorem{prop}[ter]{{\it \textbf{Proposition}}}
\newtheorem{lem} [ter]{{\it \textbf{Lemma}}}
\newtheorem{obs}[ter]{{\it \textbf{Remark}}}
\newtheorem{defi}[ter]{{\it \textbf{Definition}}}
\newtheorem{cor}[ter]{{\it \textbf{Corollary}}}

\title{Propagation in the Fisher-KPP equation with Mixed Operator.}
\author[1]{Bego\~na Barrios}
\author[2]{Bryan Pichucho}
\author[2]{Alexander Quaas}

\affil[1]{Departamento de Análisis Matemático, Universidad de La Laguna, Tenerife, España}
\affil[2]{Departamento de Matemática, Universidad Técnica Federico Santa María, Valparaíso, Chile}
\date{}

\begin{document}
\maketitle

\begin{abstract}

Our investigation focuses on the asymptotic spreading behavior of the Fisher-KPP equation with a mixed local-nonlocal operator in the diffusion \cite{cabre_influence_2013} to the setting of mixed diffusion, which involves both the classical and the fractional Laplacian in order to analyze the long-time dynamics of the equation. A key step in our approach involves the construction and detailed study of the heat kernel associated with the mixed operator, which we use to develop a theory of mild solutions and establish a comparison principle in suitable weighted function spaces.\\
 This framework allows us to rigorously establish the non-existence of traveling waves and characterize the large-time spreading rate of solutions. We show that the influence of the fractional Laplacian dominates over the classical Laplacian, especially in the initial layer, where it dictates the exponential propagation rate and the thickness of the solution tails.

\end{abstract}


\section{Introduction.}

Reaction-diffusion equations represent a cornerstone in the mathematical modeling of various physical, biological, and ecological phenomena where the interplay between diffusion and reaction mechanisms governs the system's dynamics. Among these, the Fisher-KPP equation stands as a paradigmatic model that has been extensively studied since its introduction in the 1930s to describe the propagation of advantageous genes in a population and the spatial spread of biological species. \\
In the study of reaction–diffusion phenomena, one of the most fundamental models is the one-dimensional equation \eqref{eq_kolmo} which was analyzed by Kolmogorov, Petrovsky and Piskunoff in \cite{kolmogorov_etude_1937},
\begin{eqnarray}\label{eq_kolmo}
    u_t =D u_{xx}+f(u),
\end{eqnarray}
which describes the evolution of a concentration $u$ of a single substance in one spatial dimension  and $D>0$ is a diffusion factor depending on the units chosen for the model. A classic instance of this model is obtained by setting $f(u)=u(1-u)$ and yields the Fisher's equation (see, \cite{fisher_wave_1937}), originally introduced  to capture the spreading of biological populations. Under suitable data $0\leq u_0(x)\leq 1$ that decays fast at infinity, it is well known that solutions to (\ref{eq_kolmo}) eventually behave like a traveling wave moving at a fixed speed.  In higher dimensions $N\geq 1$, the problem becomes ( for $D=1$)
\begin{eqnarray}\label{eq_AW}
    u_t-\Delta u=f(u) \quad \text{ in } (0,+\infty)\times \R^N,
\end{eqnarray}
that has been studied by Aronson and Weinberger in \cite{aronson_multidimensional_1978}. They showed that solutions propagate linearly in time when starting from a compactly supported initial condition, as follows
\begin{ter}[\cite{aronson_multidimensional_1978}]
    Let $u$ be a solution of (\ref{eq_AW}) with $u(0, \cdot) \not \equiv 0$ compactly supported in $\mathbb{R}^N$ and satisfying $0 \leq u(0, \cdot) \leq 1$. Let $c_*=2 \sqrt{f^{\prime}(0)}$. Then,\\
a) if $c>c_*$, then $u(t, x) \rightarrow 0$ uniformly in $\{|x| \geq c t\}$ as $t \rightarrow+\infty$,\\
b) if $c<c_*$, then $u(t, x) \rightarrow 1$ uniformly in $\{|x| \leq c t\}$ as $t \rightarrow+\infty$.
\end{ter}
In addition, problem (\ref{eq_AW}) admits planar traveling wave solutions connecting $0$ and $1$, that is, solutions of the form $u(t,x)=\varphi(x \cdot e+ct)$ with 
\begin{eqnarray*}
    -\varphi''+c\varphi'=f(\varphi) \text{ in } \R, \quad \varphi(-\infty)=0, \, \varphi(+\infty)=1.
\end{eqnarray*}
Moreover, recent refinements in higher dimensions by Roquejoffre et al. \cite{roquejoffre_sharp_2019}, have shown that, under radially symmetric and compactly supported initial data, the Fisher-KPP equation admits asymptotic front profile solutions with a direction-dependent logarithmic delay characterized by a Lipschitz function on the unit sphere. More specifically, they showed that the solution $u(t,x)$ converges uniformly to a planar traveling front modulated by a displacement that depends on the direction of propagation.

This asymptotic traveling wave behavior has been generalized in many interesting ways, with remarkable applications, and there is a huge mathematical literature devoted to them (see, e.g., \cite{aronson_nonlinear_1975, aronson_multidimensional_1978, hamel_fast_2010, berestycki_analysis_2005}). 

In contrast to the classical diffusion case, where traveling waves emerge with a linear propagation speed, the scenario governed solely by fractional diffusion has a markedly different behavior. In \cite{cabre_influence_2013}, Cabré and Roquejoffre considered the fractional diffusion equation
\begin{eqnarray}\label{eq_NLocal}
    u_t+(-\Delta)^s u=f(u), \quad \text{ in } (0,+\infty)\times \R^N,
\end{eqnarray}
which correspond to \eqref{eq_kolmo} in the case where the Laplacian is replaced by $(-\Delta)^s$, $s\in (0,1)$, the fractional Laplacian. This operator can be understood through two equivalent formulations. For a function $u: \mathbb{R}^N \to \mathbb{R}$, with suitable regularity, the integral representation is given by
\begin{eqnarray*}
(-\Delta)^{s} u(x) & :=&C_{N, s} \cdot \text { P.V. } \int_{\R^N} \frac{u(x)-u(y)}{|x-y|^{N+2 s}} d y \\
& =&C_{N, s} \cdot \lim _{\varepsilon \rightarrow 0^{+}} \int_{\{|x-y| \geq \varepsilon\}} \frac{u(x)-u(y)}{|x-y|^{N+2 s}} d y,
\end{eqnarray*}
where $C_{N,s} = \frac{2^{2s-1}2s\Gamma((N+2s)/2)}{\pi^{N/2}\Gamma(1-s)}$ is a normalization constant. Alternatively, for a function $u$ in the Schwartz's class $\mathcal{S}(\R^N)$ we define its Fourier transform as 
\begin{equation*}
    \mathfrak{F}(u)(\xi)=\dfrac{1}{(2\pi)^{N/2}}\int_{\R^N} e^{-i x\cdot \xi} u(x) \, dx,
\end{equation*}
then one has 
\begin{equation*}
(-\Delta)^s u(x) = \mathfrak{F}^{-1}(|\xi|^{2s}\mathfrak{F}\, u)(x).
\end{equation*}
The integral representation highlights the nonlocal nature of the operator, while the Fourier representation facilitates spectral analysis and provides a powerful tool for analytical investigations. For more details, see \cite{silvestre_regularity_2007, stinga_users_2019, biagi_global_2025}. In \cite{cabre_influence_2013} the authors concluded that there is no traveling wave behavior of \eqref{eq_NLocal}  as $t\to \infty$, and the level sets of the solution spread at an exponential rate over time (see, e.g., \cite{coulo2014}). Here, 
\begin{equation}\label{hyp_f}
    f \in C^1([0,1]) \text { is concave, } f(0)=f(1)=0 \text {, and } f^{\prime}(1)<0<f^{\prime}(0) \text {. }
\end{equation}
This rapid propagation is a natural consequence of the nonlocal character of the fractional Laplacian, which accounts for long-range dispersal even in linear diffusion processes.

The contrasting behaviors between classical diffusion (traveling waves with linear propagation) and fractional diffusion (no traveling waves with exponential propagation) raise a fundamental question: How does a system behave when both local and non local diffusion mechanisms are simultaneously present? This mixed scenario has significant implications for ecological systems where both regular diffusion and occasional long-range dispersal occur. Our work addresses this question by examining the Fisher-KPP equation with a mixed operator \eqref{eq_mix}. That is, we study the asymptotic behavior of the solution $u=u(t,x)$ to the Cauchy problem for the Fisher-KPP equation with a mixed diffusion operator, that is,
\begin{eqnarray}\label{eq_mix}
    \left\{\begin{aligned}
u_t+\mathcal{L} u & =f(u) & & \text { in } (0,+\infty)\times\mathbb{R}^N , \\
u(0, \cdot) & =u_0 & & \text { in } \mathbb{R}^N, \quad 0 \leq u_0 \leq 1,
\end{aligned}\right.
\end{eqnarray}
where the operator $\La$ is defined as
\begin{equation*}
\La u := -\Delta u + (-\Delta)^s u, \quad s \in (0,1),
\end{equation*}
where $f$ satisfies \eqref{hyp_f}. 
For further details on the model, see \cite{aronson_nonlinear_1975, aronson_multidimensional_1978, hamel_fast_2010}. 

The study of mixed local-nonlocal operators has gained significant attention in recent years due to their capacity to model complex diffusion processes that involve both short-range and long-range interactions \cite{dipierro_qualitative_2025, biagi_mixed_2022, biagi_faber-krahn_2023, biagi_global_2025, song_parabolic_2007, ashok_kumar_strict_2025}. While the classical Laplacian accounts for local diffusion processes following Brownian motion, the fractional Laplacian introduces non-local effects, representing anomalous diffusion with Lévy flight characteristics \cite{aronson_multidimensional_1978, cabre_influence_2013}. This combination provides a more realistic framework for modeling natural phenomena where different scales of interaction coexist, such as in population dynamics with both local dispersal and occasional long-distance migration, or in certain physical systems exhibiting mixed diffusion \cite{dipierro_description_2021}.

It is worth to mention a recent work concerning a generalized mixed operator, where the non-local component is more general than the fractional Laplacian. It is due to Gonzálvez et al. \cite{gonzalvez_nonlocal_2025}, who develop a Widder-type theory for a broad class of Lévy-type operators, including those with both local and non-local parts. In particular, they show that the heat kernel associated with such mixed operators is comparable to the heat kernel of the purely non-local part, under general growth assumptions on the non-local kernel.  We refer to (Proposition 5.1, \cite{gonzalvez_nonlocal_2025}) as a general alternative to the result given in \cite{song_parabolic_2007}, which is specific to fractional Laplacians.


\subsection{Methodology and organization.}

Through a semigroup-based approach, we provide a rigorous framework for analyzing the asymptotic properties of solutions. Our approach follows the ideas of \cite{cabre_influence_2013}  with a careful characterization of the heat kernel associated with the mixed operator, which serves as the foundation for establishing comparison principles and bounds on the corresponding semigroup coupled with the behavior of the nonlinearity.

In particular, we establish in Theorem \ref{prop:no_traveling_waves} and Theorem \ref{Main_Thconvergence} that the nonlocal effects dominate that means; similar results as those found in \cite{cabre_influence_2013}  can be established for the mixed operator, i.e, there is no traveling wave exists for the KKP type nolinearity and the speed of spreding has sharp convergence to the steady states ( $0$ and $1$)  in conical space domains with exponential speed in time.


The paper is organized as follows. In Section 2, we present the necessary preliminaries on the mixed operator $\mathcal{L}$, characterizing its heat kernel, establishing key properties and we state our main results. Section 3 is devoted to the analysis of the semilinear problem, including the existence, uniqueness, and comparison principles for solutions. Finally,  Section 4, we derive bounds on the associated semigroup and establish our main results regarding the non-existence of traveling waves and the asymptotic spreading rates.



Throughout this paper, the constant $C$ denotes a positive constant whose value may vary from line to line.



\section{The mixed operator $\La =-\Delta+(-\Delta)^s$.}

In this section, we establish the necessary mathematical framework for studying the mixed local-nonlocal operator $\mathcal{L} $. For full details, we refer to Biagi et al. \cite{biagi_global_2025}. We first recall essential properties of the fractional Laplacian, then construct the appropriate function spaces, and finally develop the heat semigroup theory for the mixed operator. 

\subsection{Semigroup characterization of the fractional Laplacian.}

A central component of our analysis is the theory of strongly continuous semigroups, which provides the mathematical foundation for understanding the evolution of solutions to equations involving the mixed operator. 

\begin{defi}[\cite{grigoryan_heat_nodate} Strongly Continuous Semigroup] \label{def_strongly_semigroup}
Let $H$ be a Hilbert space and $\mathcal{J}$ a non-negative definite, self-adjoint operator in $H$. A family of bounded linear operators $(e^{-t \mathcal{J}})_{t\geq0}:=(P(t))_{t \geq 0} \subset \mathcal{B}(H)$ is a strongly continuous semigroup on $H$, If
\begin{enumerate}
\item $P(0) = I$ (identity operator).
\item $P(t+s) = P(t)P(s)$, for all $t, s \geq 0$ (semigroup property).
\item For any $t\geq0$ and $f\in H$, we have 
\begin{eqnarray*}
   \lim _{\tau \rightarrow t} P(\tau) f=P(t) f, \quad \text { in } H.
\end{eqnarray*}
\item For all fixed $t>0$ and $f \in H$, we have $P(t) f \in H$, and
$$
\frac{d}{d t}(P(t) f)=\lim _{h \rightarrow 0} \frac{P(t+h) f-P(t) f}{h}=- \mathcal{J}(P(t) f), \, \, \text { in } H.
$$
\end{enumerate}
\end{defi}
The classical theory of semigroups \cite{pazy_semigroups_2012, grigoryan_heat_nodate, cazenave_introduction_1998} ensures that for appropriate operators, these abstract semigroups admit concrete representations via integral kernels, which we now develop for our specific operators.\\ 
Lets define the fractional Sobolev space $H^{s}\left(\R^N\right)=\left\{u \in L^{2}\left(\R^N\right):|\xi|^{2 s} \mathfrak{F}(u) \in L^{2}\left(\R^N\right)\right\}$. Thus, following (\cite{biagi_global_2025}, Sec. 2), we define 
\begin{eqnarray*}
    \mathcal{B}_{s}: H^{s}\left(\R^N\right) \subseteq L^{2}\left(\R^N\right) \rightarrow L^{2}\left(\R^N\right), \quad \mathcal{B}_{s}(u)=\mathfrak{F}^{-1}\left(|\xi|^{2 s} \mathfrak{F}(u)\right).
\end{eqnarray*}
We have the following.
\begin{prop}[\cite{biagi_global_2025} Properties of $\mathcal{B}_s$]
The operator $\mathcal{B}_s$ is
\begin{enumerate}
\item Self-adjoint on $L^2(\mathbb{R}^N)$.
\item Non-negative definite.
\item Densely defined (since $\mathcal{S} \subset H^s(\mathbb{R}^N)$, where $\mathcal{S}$ denotes the Schwartz space).
\end{enumerate}
Moreover, $\mathcal{B}_s(u) = (-\Delta)^s u$ for all $u \in \mathcal{S}$, so $\mathcal{B}_s$ is indeed a realization of $(-\Delta)^s$ on $L^2(\mathbb{R}^N)$.
\end{prop}
 Consequently, by Definition \ref{def_strongly_semigroup},  the operator $-\mathcal{B}_{s}$ generates a strongly continuous semigroup on the Hilbert space $L^{2}\left(\R^N\right)$.

This semigroup is referred to as the heat semigroup of \(-(-\Delta)^s\) and is denoted by \(\{e^{-t(-\Delta)^s}\}_{t\geq 0}\). Moreover, by combining property $4$ of Definition \ref{def_strongly_semigroup} with the definition of \(\mathcal{B}_s\) and utilizing the Fourier transform, one can show that for every \(t>0\) the operator \(e^{-t(-\Delta)^s}\) can be expressed as an integral operator on \(L^2(\mathbb{R}^N)\) with a convolution kernel. Indeed, for $f\in L^2(\R^N)$ we consider the evolution equation
\begin{equation*}
\begin{cases}
 u_t  +(-\Delta)^s u=0, \text{ in } (0,+\infty)\times \R^N, \\
u(x,0) = f(x), \text{ in } \R^N.
\end{cases}
\end{equation*}
Applying the Fourier transform to both sides of the before equation, we get
\begin{eqnarray*}
     \mathfrak{F}(u_t)(t,\xi) +|\xi|^{2s}\mathfrak{F}(u)(t,\xi)=0,
\end{eqnarray*}
and
\begin{eqnarray*}
    \mathfrak{F}(u)(0,\xi) =\mathfrak{F}(f)(\xi).
\end{eqnarray*}
We observe that taking the Fourier transform of the evolution equation leads to a first-order ODE in time for each fixed $\xi\in \R^N$. The solution  is 
\begin{eqnarray*}
    \mathfrak{F}(u)(t,\xi) =\mathfrak{F}(f)(\xi)e^{-t|\xi|^{2s}}.
\end{eqnarray*}
Now, taking the inverse Fourier transform, we have 
\begin{eqnarray*}
    u(t,x) = \mathfrak{F}^{-1}(\mathfrak{F}(f)(\cdot)e^{-t|\cdot|^{2s}})(x).
\end{eqnarray*}
So, using the properties of the Fourier transform, we obtain 
\begin{eqnarray*}
    u(t,x) &=& \mathfrak{F}^{-1}(\mathfrak{F}(u)(t,\cdot))(x) \\
   &=& \mathfrak{F}^{-1}(\mathfrak{F}(f)(\cdot)e^{-t|\cdot|^{2s}})(x) \\
   &=& \frac{1}{(2\pi)^{N/2}}\int_{\R^N} \mathfrak{F}(f)(\xi)e^{-t|\xi|^{2s}}e^{ix\cdot\xi}d\xi\\
   &=& \frac{1}{(2\pi)^N}\int_{\R^N} \left(\int_{\R^N} f(y)e^{-iy\cdot\xi}dy\right)e^{-t|\xi|^{2s}}e^{ix\cdot\xi}d\xi \\
   &=& \int_{\R^N} f(y)\left(\frac{1}{(2\pi)^N}\int_{\R^N} e^{i(x-y)\cdot\xi}e^{-t|\xi|^{2s}}d\xi\right)dy.
\end{eqnarray*}
Note that the term in parentheses is precisely $p^{(s)}(t,x-y)$, where:
\begin{eqnarray}\label{kernel_fourier_fract}
p^{(s)}(t,z) := \frac{1}{(2\pi)^{N/2}}\mathfrak{F}^{-1}(e^{-t|\cdot|^{2s}})(z).    
\end{eqnarray}
 Therefore, 
 \begin{eqnarray*}
     u(t,x) = \int_{\R^N} p^{(s)}(t,x-y)f(y)dy = (p^{(s)}(t,\cdot) * f)(x).
 \end{eqnarray*}

This function $(t, z) \mapsto p^{(s)}(t,z)$  represents the fundamental solution, or heat kernel, corresponding to the fractional operator $-(-\Delta)^s$. This kernel exhibits several key properties that are essential for the analysis of fractional diffusion processes (see, e.g., \cite{cabre_influence_2013, chen_global_2011, chen_heat_2003} for complete proofs).
\begin{itemize}
    \item[(1)]  $p^{(s)} \in C^{\infty}\left(\mathbb{R}^{+} \times \R^N\right)$ and $p^{(s)}>0$.
\item[(2)] For all $x \in \R^N$ and $t>0$, we have

$$
p^{(s)}(t,x)=p^{(s)}(t,-x) \quad \text { and } \quad p^{(s)}(t,x)=\frac{1}{t^{1 /(2 s)}}p^{(s)}\left(1,t^{-1 /(2 s)} x\right)
$$

\item[(3)] For all $x \in \R^N$ and $t>0$, we have

$$
\int_{\R^N} p^{(s)}(t,x) d x=1
$$

\item[(4)] For all $x \in \R^N$ and $t, \tau>0$, we have

$$
\int_{\R^N} p^{(s)}(t,x-y) p^{(s)}(\tau,y) d y=p^{(s)}(t+\tau,x)
$$

\item[(5)] There exists $B \geq 1$ such that  for all $x \in \R^N$ and all $t>0$,

 \begin{equation}\label{est_frac_kernel}
     \dfrac{B^{-1}}{t^{\frac{N}{2s}}\left(1+|t^{-\frac{1}{2s}}x|^{N+2s}\right)} \leq p^{(s)}(t,x) \leq \dfrac{B}{t^{\frac{N}{2s}}\left(1+|t^{-\frac{1}{2s}}x|^{N+2s}\right)}.
 \end{equation}
\end{itemize}
\subsection{The heat kernel associated with the mixed operator $\mathcal{L}$.}
 Having established the fundamental properties of the fractional Laplacian $(- \Delta)^s$, we now briefly discuss the heat semigroup and corresponding global heat kernel generated by the mixed operator $\mathcal{L}=-\Delta+(-\Delta)^{s}$. The characterization of this heat kernel plays a role in our subsequent framework, as it enables us to formulate the concept of mild solutions to the Cauchy problem \eqref{eq_mix}, which we will formally introduce in Definition \ref{def_mildsol}.\\
 Let us introduce $H^2(\mathbb{R}^N)$ as the standard Sobolev space of functions with square-integrable second derivatives, traditionally denoted as $W^{2,2}(\mathbb{R}^N)$. Then, following (\cite{biagi_global_2025}, Section 2), we define the operator
\begin{eqnarray*}
    \mathcal{P}: H^{2}\left(\R^N\right) \rightarrow L^{2}\left(\R^N\right), \quad \mathcal{P}(u)=\mathcal{A}(u)+\mathcal{B}_{s}(u)=\mathfrak{F}^{-1}\left(|\xi|^{2} \mathfrak{F}(u)\right)+\mathfrak{F}^{-1}\left(|\xi|^{2 s} \mathfrak{F}(u)\right),
\end{eqnarray*}
where 
\begin{eqnarray*}
    \mathcal{A}: H^{2}\left(\R^N\right) \rightarrow L^{2}\left(\R^N\right), \quad \mathcal{A}(u)=\mathfrak{F}^{-1}\left(|\xi|^{2} \mathfrak{F}(u)\right),
\end{eqnarray*}
 is a densely defined, positive and self-adjoint operator with $\mathcal{A}(u)=-\Delta u$ for all $u \in \mathcal{S} \subseteq H^{2}\left(\R^N\right)$ (see \cite{evans_partial_2022}, Section 4.3, for instance). Thus, we have the following result.
\begin{prop}[\cite{biagi_global_2025} Properties of $\mathcal{P}$]
The operator $\mathcal{P} = \mathcal{A} + \mathcal{B}_s$ is
\begin{enumerate}
\item Densely defined on $L^2(\mathbb{R}^N)$.
\item Positive definite.
\item Self-adjoint.
\end{enumerate}
Moreover, $\mathcal{P}(u) = \mathcal{L}u$ for all $u \in \mathcal{S} \subset H^2(\mathbb{R}^N)$, so $\mathcal{P}$ is indeed a realization of $\mathcal{L}$ on $L^2(\mathbb{R}^N)$.
\end{prop}
One can then exploit once again Definition \ref{def_strongly_semigroup}, we establish that the operator $-\mathcal{P}$ generates a strongly continuous semigroup on the Hilbert space $L^2(\mathbb{R}^N)$. We denote this semigroup by
\begin{eqnarray*}
    \left(e^{-t \mathcal{L}}\right)_{t \geq 0}:=\left(T_t\right)_{t \geq 0},
\end{eqnarray*}
in this way the family $(T_t)_{t\geq 0}$ satisfies  the same properties of the Definition \ref{def_strongly_semigroup}, with $-\mathcal{P}$ in place of $-\mathcal{B}_{s}$. This semigroup is called the heat semigroup of $-\mathcal{L}$.\\
Now, by arguing exactly as before, we see that the operator $T_t$, (for all fixed $t>0$, ) is a integral operator with convolution-type kernel. In fact, as previously discussed, for $u_0 \in L^2(\R^N)$, if we consider the corresponding evolution equation given by 
\begin{equation}\label{eq_mixhomo}
\begin{cases}
u_t + \La u=0 & \text{in } (0,+\infty)\times\mathbb{R}^N  , \\
u(0,x) = u_0(x) & \text{in } \mathbb{R}^N,
\end{cases}
\end{equation}
 taking the Fourier transform in the spatial variable of \eqref{eq_mixhomo}, we obtain
\begin{equation}\label{eq:fourier_heat}
 \mathfrak{F}({u}_t)(\xi,t) +(|\xi|^2 + |\xi|^{2s}) \mathfrak{F}({u})(\xi,t)=0.
\end{equation}

\noindent Since, the solution of this ODE in Fourier space is $\mathfrak{F}({u})(\xi,t) = e^{-(|\xi|^2 + |\xi|^{2s})t} \mathfrak{F}({u}_0)(\xi)$, we obtain
\begin{eqnarray*}
    u(t,x) =\mathfrak{F}^{-1} (\mathfrak{F}({u}_0)(\cdot)e^{-(|\cdot|^2 + |\cdot|^{2s})t})(x)=(\He(t,\cdot)* u_0)(x)=\int_{\R^N} \He(t,x-y)u_0(y)\,dy, 
\end{eqnarray*}
where, for all $z \in \R^N$ and $t>0$, 
\begin{equation}\label{kernel_fourier_mix}
\He(t,z):=\frac{1}{(2 \pi)^{N / 2}} \mathfrak{F}^{-1}\left(e^{-t\left(|\xi|^{2}+|\xi|^{2 s}\right)}\right)(z)=\frac{1}{(2 \pi)^N} \int_{\R^N} e^{i\langle z, \xi\rangle-t\left(|\xi|^{2}+|\xi|^{2 s}\right)} d \xi.
\end{equation}
On the other hand, by exploiting (\ref{kernel_fourier_fract}) (jointly with the explicit expression of $\mathfrak{F}^{-1}\left(e^{-t|\xi|^{2}}\right)$ and the properties of the Fourier transform), we obtain
\begin{eqnarray}\label{kernel_mix}
    \He(t,z) 
& =&\frac{1}{(2 \pi)^{N / 2}} \mathfrak{F}^{-1}\left((2 \pi)^{N / 2} \mathfrak{F}\left(p^{(2)}(t,\cdot)\right) \cdot(2 \pi)^{N / 2} \mathfrak{F}\left(p^{(s)}(t,\cdot)\right)\right)(z) \nonumber\\
& =&\mathfrak{F}^{-1}\left((2 \pi)^{{N} / 2} \mathfrak{F}\left(p^{(2)}(t,\cdot)\right) \cdot \mathfrak{F}\left(p^{(s)}(t,\cdot)\right)\right)(z) \nonumber\\
& =&\left(p^{(2)}(t,\cdot) * p^{(s)}(t,\cdot)\right)(z)\nonumber\\
&=&\int_{\mathbb{R}} p^{(2)}(t,z-y)p^{(s)}(t,y)\,dy,
\end{eqnarray}
where $p^{(2)}(t,\cdot)$ is the classical Gaussian heat kernel of $\Delta$, given by
\begin{eqnarray}\label{gauss_ker}
  p^{(2)}(t,z):=\frac{1}{(4 \pi t)^{{N} / 2}} e^{-|z|^{2} /(4 t)}. 
\end{eqnarray}
Summing up, we conclude that
\begin{equation}\label{kernel_mix__gauss_frac}
\He(t,z)=\frac{1}{(4 \pi t)^{{N} / 2}} \int_{\R^N} e^{-|z-y|^{2} /(4 t)} p^{(s)}(t,y) \, d y, \quad z \in \R^{N}, t>0. 
\end{equation}

This function $(t, z) \mapsto \He(t,z)$ constitutes the heat kernel associated with the operator $-\mathcal{L}$, exhibiting properties comparable to those of $p^{(s)}$. For subsequent analysis, we consolidate these essential properties in the following result. These characteristics emerge naturally from the explicit representation of $\He$ detailed in  (\ref{kernel_fourier_mix}) - (\ref{kernel_mix__gauss_frac}).

\begin{prop}\label{P1}
The heat kernel $\He$ satisfies the defined by \eqref{kernel_mix} satisfies the following properties
\begin{enumerate}
\item [$\bullet$] $0< \He(t,x)\in C(  (0,+\infty)\times \R^{N})$.
\item [$\bullet$] $\int_{\mathbb{R}^N} \He(t,x)\,dx = 1$ for all $t > 0$.
\item[$\bullet$]  For any $u_0 \in L^1(\mathbb{R}^N) \cap L^\infty(\mathbb{R}^N)$, the function $u(t,x) = \int_{\mathbb{R}^N} \He(t,x-y)u_0(y)\,dy$, 
solves \eqref{eq_mixhomo}.
\item[$\bullet$]  Upper estimate: there exists a constant $C>0$ such that
\begin{equation*}
    0<\He(t,x) \leq C t^{\frac{1}{2s}} \text{ for every  } x\in \R^{N} \text{ and } t>0.
\end{equation*}
\item[$\bullet$]   If $t>1$ and $|x|>M t^{\frac{1}{2 s}}$, with $M>0$,
\begin{eqnarray*}
    \frac{\alpha t}{2|x|^{N+2 s}}<\mathcal{H}(t,x)<\frac{2 \alpha t}{|x|^{N+2 s}},
\end{eqnarray*}
where $\alpha=2^{N+2s} \pi^{\frac{N}{2}-1}s \Gamma(N/2+s)\Gamma(s)$ (see for instance \cite{dipierro_qualitative_2025}).

\end{enumerate}
\end{prop}

Moreover, it also has lower and upper bounds for the kernel $\He (t,x)$, given in the next.
\begin{prop}[\cite{song_parabolic_2007} Parabolic Harnack inequality]\label{P2}
There exists a positive constant $C$ such that
\begin{equation}\label{est_mix_kernel}
    C^{-1}\, q_1(t,x) \leq \He(t,x) \leq C\, q_2(t,x) .
\end{equation}
\end{prop}
Here,
\begin{eqnarray*}
    \hat{p}^{(2)}(t,x)=(4 \pi t)^{-N/2} \exp\left(-\frac{|x|^2}{t}\right), \qquad
    \tilde{p}^{(2)}(t,x)=(4 \pi t)^{-N/2} \exp\left(-\frac{|x|^2}{16t}\right),
\end{eqnarray*}
$$
q_1(t,x)= \begin{cases}\hat{p}^{(2)}(t,x), & |x|^2<t<|x|^{2s} \leq 1, \\ \max \left(\hat{p}^{(2)}(t,x), p^{({s})}(t,x)\right), & t<|x|^2 \leq 1, \\ p^{(2)}(t,x), & |x|^{2s} \leq t \leq 1, \\ p^{(s)}(t,x), & t \geq 1 \text { or }|x| \geq 1,\end{cases}
$$

and
\begin{equation} \label{upper_est_mix_op}
    q_2(t,x)= \begin{cases}p^{(2)}(t,x), & |x|^2<t<|x|^{2s} \leq 1, \\ \max \left(\tilde{p}^{(2)}(t,x), p^{(s)}(t,x)\right), & t<|x|^2 \leq 1, \\ p^{(2)}(t,x), & |x|^{2s} \leq t \leq 1, \\ p^{(s)}(t,x), & t \geq 1 \text { or }|x| \geq 1 .\end{cases}
\end{equation}

The previous situation can be summarized as follows. 
\begin{center}
    \includegraphics[scale=0.15]{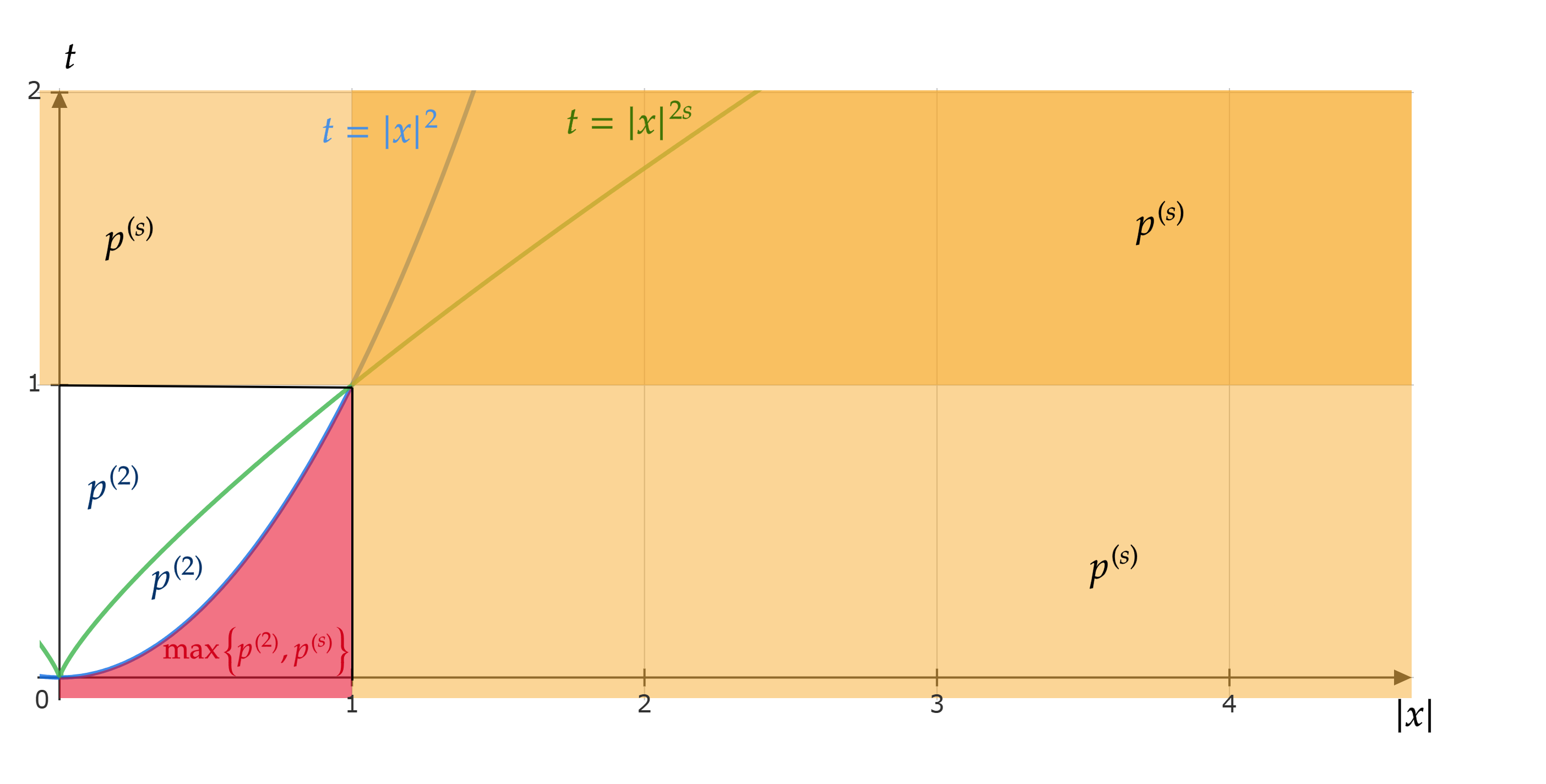}
\end{center}

\subsection{Main results.}

Our investigation provides a comprehensive analysis of how the interplay between local and nonlocal diffusion mechanisms affects the asymptotic propagation of solutions to the Fisher-KPP equation. We establish three main results that characterize the qualitative behavior of solutions.

First, we prove that the presence of the fractional Laplacian term fundamentally alters the propagation dynamics, leading to the non-existence of traveling wave solutions.

\begin{ter}[Non-existence of Traveling Waves]\label{prop:no_traveling_waves}
     Let $N\geq 1$, $ s \in(0,1), f$ satisfy (\ref{hyp_f}), and let $\He$ be a kernel satisfying the properties of Proposition \ref{P1}. Then, there exists no nonconstant planar traveling wave solution of (\ref{eq_mix}). That is, all solutions of (\ref{eq_mix}) taking values in $[0,1]$ and of the form $u(t, x)=\varphi(x+ct)$, for some vector $c\in \R^{N}$, are identically 0 or 1. Equivalently, the only solutions $\varphi: \mathbb{R}^N \rightarrow[0,1]$ of
\begin{eqnarray*}
   -\Delta \varphi + (-\Delta)^s \varphi+c \cdot \nabla \varphi=f(\varphi) \quad \text { in } \mathbb{R}^N,
\end{eqnarray*}
are $\varphi \equiv 0$ and $\varphi \equiv 1$.
\end{ter}

Next, we characterize the asymptotic spreading behavior of solutions with appropriate initial conditions, showing that the propagation occurs at an exponential rate determined by the fractional diffusion term.

\begin{ter}[Asymptotic Spreading]\label{Main_Thconvergence}
Let $N\geq1$, $ s\in(0,1)$, $f$ satisfy (\ref{hyp_f}), and $\He$ be a kernel satisfying the properties of Proposition \ref{P1}.\\
Let $\sigma_*=\frac{f^{\prime}(0)}{N+2 s}$. Let u be a solution of (\ref{eq_mix}), where $u_0 \not \equiv 0,0 \leq u_0 \leq 1$ is measurable, and
$$
u_0(x) \leq C|x|^{-N-2 s} \text { for all } x \in \mathbb{R}^N
$$
and for some constant $C$. Then,
\begin{itemize}
    \item [$a)$] if $\sigma>\sigma_*$, then $u(t, x) \rightarrow 0$ uniformly in $\left\{|x| \geq e^{\sigma t}\right\}$ as $t \rightarrow+\infty$,
    \item [$b)$]if $\sigma<\sigma_*$, then $u(t, x) \rightarrow 1$ uniformly in $\left\{|x| \leq e^{\sigma t}\right\}$ as $t \rightarrow+\infty$.
\end{itemize}
\end{ter}

These results collectively demonstrate that the fractional component dominates the long-term behavior of solutions, dictating both the propagation rate and the profile of the invasion front. 
More precisely,  the Laplacian $\Delta$ only modifies short-range interactions. For $t \geq 1$, Proposition \ref{P2} ensures that
    \begin{eqnarray*}
        \He(t,x) \sim \text{ Fractional kernel } p^{(s)}(t,x),
    \end{eqnarray*}
   preserving the propagation rate given in  \cite{cabre_influence_2013}, which is $\sigma_*=\frac{f'(0)}{N+2s}$.


\section{Main properties related with the semilinear problem.}
In this section, we collect certain aspects concerning the existence and comparison of solution of the semilinear problem (\ref{eq_mix}). First of all, we introduce the following.

\begin{defi}[Mild solution] \label{def_mildsol}
    Let $X:=L^{1}\left(\mathbb{R}^N\right) \cap L^{\infty}\left(\mathbb{R} ^N\right)$. Given $u_{0} \in X$ a function $u \in C([0, T) ; X)$ is said to be a mild solution of (\ref{eq_mix}) if
\begin{equation}\label{eq:mild}
 u(t,x)=\int_{\mathbb{R}^{N}} \He(t,x-y) u_{0}(y)\, d y+\int_{0}^{t} \int_{\mathbb{R}^{N}} \He(t-r,x-y) f(u(r,y))\, d y d r, 
\end{equation}
for almost every $x \in \mathbb{R}^N$ and all $t \in (0,T)$, where $\He$ defined in (\ref{kernel_mix}), is the fundamental solution  of the linear equation (\ref{eq_mixhomo}).
\end{defi}
\begin{ter}
    For each $u_{0} \in X$ there exists a unique mild solution $u \in C((0, T) ; X)$ of (\ref{eq_mix}). 
\end{ter}
\begin{proof}
We denote by  $M:=\left\|u_{0}\right\|_{X}=\left\|u_{0}\right\|_{L^{1}\left(\mathbb{R}^{N}\right)}+\left\|u_{0}\right\|_{L^{\infty}\left(\mathbb{R}^{N}\right)}$ and for $T_{0}>0$ fixed but arbitrary that will be chosen later, we consider the space

$$
E:=\left\{v \in C\left(\left(0, T_{0}\right) ; X\right): \|v\|_{E}:=\sup _{0<t<T_{0}}\|v(t,\cdot)\|_{X} \leq 4 M\right\}.
$$

We now define the operator
\begin{equation}\label{op}
    \Phi(v)(t)=\int_{\mathbb{R}^N} \He(t,x-y) u_{0}(y) \, d y+\int_{0}^{t} \int_{\mathbb{R}^N} \He(t-r,x-y) f(v(r,y)) \, d y d r.
\end{equation}
We want to prove that if $T_{0}$ is small then $\Phi: E \rightarrow E$ is contractive, and thus has a unique fixed point. For that let us first observe that, by conservation of the mass, 
\begin{eqnarray*}
    \left|\int_{\mathbb{R}^N} \He(t,x-y)u_0(y)dy\right| \leq \|u_0\|_{L^\infty(\R^{N})} \int_{\mathbb{R}^N} \He(t,x-y)dy 
= \|u_0\|_{L^\infty(\R^{N})} \cdot 1 
= M.
\end{eqnarray*}
Moreover, since (\ref{hyp_f}) implies that there exist a $c>0$ such that $f(v)\leq c v$, if $v\in E$, it is also clear that
\begin{eqnarray*}
    \left|\int_{\mathbb{R}^N} \He(t-r,x-y)f(v(r,y))dy\right| 
\leq c \|v(r,\cdot)\|_{L^\infty(\R^{N})} \int_{\mathbb{R}^N} \He(t-r,x-y)dy 
= 4 \, c M .
\end{eqnarray*}
Thus, if we take for instance $T_0\leq \frac{1}{4c}$,
\begin{eqnarray*}
    \|\Phi(v)(t)\|_{L^{\infty}\left(\mathbb{R}^N\right)} \leq M +\int_{0}^{t} 4cM d \tau 
 \leq M+4 c M T_{0}
\leq 2 M.
\end{eqnarray*}
Similarly we also have that
\begin{eqnarray*}
    \|\Phi(v)(t)\|_{L^{1}\left(\mathbb{R}^N\right)}  \leq\left\|u_{0}\right\|_{L^{1}\left(\mathbb{R}^N\right)}+\int_{0}^{t} c \|v(\tau)\|_{L^{1}\left(\mathbb{R}^N\right)} d \tau 
 \leq M+4c M T_{0}
\leq 2 M,
\end{eqnarray*}
provided $T_{0}$ is small enough. Hence, since we have proved that $\|\Phi(v)(t)\|_{X}\leq 4M$, it is clear that $\Phi(E) \subset E$. 
Moreover, by the Mean Value Theorem and (\ref{hyp_f}), for $v_{1}, v_{2} \in E$, we have
\begin{eqnarray*}
    \left\|\Phi\left(v_{1}\right)(t)-\Phi\left(v_{2}\right)(t)\right\|_{L^{\infty}(\R^{N})} &\leq&
    \int_{0}^{t} \int_{\mathbb{R}^{N}} \He(t-r, x-y) |f(v_1(r, y))-f(v_2(r, y))| \, d y d r\\
    &=& \int_{0}^{t} \int_{\mathbb{R}^{N}} \He(t-r, x-y) |f'(z)||v_1(r, y)-v_2(r, y)| \, d y d r\\
    & \leq& \int_{0}^{t} C \left\|v_{1}(\tau)-v_{2}(\tau)\right\|_{L^{\infty}(\R^{N})} d \tau, 
\end{eqnarray*}
where we use the fact that by  (\ref{hyp_f}), $f'(z)<f'(0)=C$, for $z>0$.\\
Likewise,
\begin{eqnarray*}
     \left\|\Phi\left(v_{1}\right)(t)-\Phi\left(v_{2}\right)(t)\right\|_{L^{1}(\R^{N})} &=&\int_{\R^{N}}\left|
    \int_{0}^{t} \int_{\R^{N}} \He(t-r,x-y) |f(v_1(r,y))-f(v_2(r,y ))| \, d y \,d r\right| dx\\
    & \leq& \int_{0}^{t} C \left\|v_{1}(\tau)-v_{2}(\tau)\right\|_{L^{1}(\R^{N})} d \tau.
\end{eqnarray*}
Therefore, for $v_{1}, v_{2} \in E$, we have proved
\begin{eqnarray*}
      \left\|\Phi\left(v_{1}\right)(t)-\Phi\left(v_{2}\right)(t)\right\|_{X}
 &\leq& 2C T_0\sup_{0<\tau<T_0} \|v_1(\tau) - v_2(\tau)\|_X \leq \dfrac{1}{2} \|v_1 - v_2\|_E,
\end{eqnarray*}
provided that $T_0\leq \frac{1}{4C}$ and consequently $\Phi$ is contractive in $E$. So we conclude by applying the Banach Fixed Point Theorem.  \\
\end{proof}


\subsection{The semigroup in $X_\gamma$}
In this subsection, we analyze several properties of the semigroup  
\begin{eqnarray}\label{def_semigroup}
    T_t u_0(x) := \int_{\mathbb{R}^{N}} \He(t,y) u_0(x-y)\,, dy=\int_{\mathbb{R}^{N}} \He(t,x-y) u_0(y)\, dy,
\end{eqnarray}
where $\He$ was given in (\ref{kernel_mix}) and $u_0\in L^{\infty}(\R^{N})$. This semigroup plays a fundamental role in our study, providing the framework for the asymptotic behavior of the solutions.  
For more background on semigroup theory and its applications, we refer to the classical works \cite{cazenave_introduction_1998, pazy_semigroups_2012, grigoryan_heat_nodate}.  
\begin{defi}
    Let \(0 \le \gamma < 2s\). We define \(X_\gamma\) as the collection of functions \(u:\R^{N}\to\R\) satisfying.
    \begin{enumerate}
        \item (Growth Condition) There exists a constant \(C>0\) such that
        \begin{equation}\label{def_1_X_gam_alt}
            |u(x)| \leq C\,(1+|x|^{\gamma}) \quad \text{for all } x \in \R^{N}.
        \end{equation}
        \item (Uniform Local Continuity) For every \(\e>0\) there exists \(\delta>0\) such that, for every \(x\in\R^{N}\) and all \(z\) with \(|z|\le \delta\),
        \begin{equation}\label{def_2_X_gam_alt}
            \frac{|u(x+z)-u(x)|}{1+|x|^{\gamma}} \le \e.
        \end{equation}
    \end{enumerate}
    That is, $X_\gamma := \left\{ u:\R^{N}\to\R \, : \, u \text{ satisfies \eqref{def_1_X_gam_alt} and \eqref{def_2_X_gam_alt}} \right\}$, equipped with the norm
    \[
    \|u\|_{X_\gamma} := \sup_{x \in \R^{N}} \frac{|u(x)|}{1+|x|^\gamma}.
    \]
\end{defi}
In (\cite{cabre_influence_2013}, Subsec. 2.1.), it is shown that
 \(X_\gamma\) is a Banach space with the inclusions \(C_0(\R^{N})\subset X_\gamma\subset C(\R^{N})\), and that the function \(w_\gamma(x)=|x|^\gamma\) belongs to \(X_\gamma\).
 
In order to analyze the long-time behavior of solutions, it is essential to establish the boundedness of the semigroup in the function space \(X_\gamma\). The following proposition ensures that \(T_t\) is a well-defined operator in \(X_\gamma\) and provides an upper bound on its norm.

\begin{prop}
For each \(t > 0\), the operator \(T_t: X_\gamma \to X_\gamma\) given in (\ref{def_semigroup}) is a bounded linear transformation. In particular, there exists a constant \(C_\gamma > 0\), independent of \(t\), such that
\begin{eqnarray*}
    \|T_t\|_{X_\gamma} \leq C_\gamma \left( 1+t^{\frac{\gamma}{2}}+t^{\frac{\gamma}{2s}} \right).
\end{eqnarray*}
\end{prop}

\begin{proof}
Let $t>0$ and $u\in X_\gamma$.
    We first show that $T_t u$ satisfies (\ref{def_1_X_gam_alt}). For that, we first notice that 
\begin{equation*}
|(T_t u)(x)| \leq C \int_{\R^{N}} |\He(t,x-y)|(1+|y|^\gamma)\,dy.
\end{equation*}
Now, motivated by (\ref{est_mix_kernel}) we estimate the integral in different regions, as follows.
\begin{itemize}
    \item In $R_1:=\{y \in \R^{N}: |x-y|^2 < t < |x-y|^{2s} \leq 1\}$, we have 
    \begin{eqnarray*}
        I_1&:=& \frac{C}{(4\pi t)^{N/2}}\int_{R_1}  e^{-\frac{|x-y|^2}{4t}}(1+|y|^\gamma)\,dy \\
        &=& \frac{C}{(4\pi t)^{N/2}}\int_{\{ 4|z|^2 t <t<|2z\sqrt{t}|^{2s}\leq1\}} e^{-|z|^2} (1+|x-2z\sqrt{t}|^{\gamma}) \, dz\\
        &\leq& \frac{C}{(4\pi t)^{N/2}} \int_{\{ 4|z|^2 t <t<|2z\sqrt{t}|^{2s}\leq1\}} e^{-|z|^2} (1+C_\gamma(|x|^{\gamma}+t^{\frac{\gamma}{2}}|z|^{\gamma}) \, dz\\
        &\leq&C_\gamma \left[(1+|x|^\gamma) \int_{\R^{N}} e^{-|z|^2} \, dz + t^{\frac{\gamma}{2}} \int_{\R^{N}} e^{-|z|^2} |z|^{\gamma}\, dz\right]\\
        &\leq& C_\gamma (1+|x|^\gamma) +C_\gamma \,t^{\frac{\gamma}{2}},
    \end{eqnarray*}
    where we use the fact that $ \int_{\R^{N}} e^{-|z|^2}\,dz = \pi^{N/2}$
and $\int_{\R^{N}} e^{-|z|^2}|z|^\gamma \,dz=C_\gamma  \, \Gamma(\frac{\gamma+N+1}{2}) < +\infty$, for $\gamma >0$.

\item In $R_2:=\{y \in \R^{N}: t < |x-y|^2 \leq 1\}$, we get 
\begin{eqnarray*}
    I_2&:=& C\int_{R_2}  \max\{\tilde{p}^{(2)}(t,x-y), p^{(s)}(t,x-y)\} (1+|y|^\gamma)\, dy\\
    &\leq& \frac{C}{(4\pi t)^{N/2}} \int_{R_2} e^{-\frac{|x-y|^2}{16t}} (1+|y|^\gamma)\, dy + C\int_{R_2} \frac{1}{t^{\frac{N}{2s}}\left(1+|t^{-\frac{1}{2s}}(x-y)|^{N+2s}\right)} (1+|y|^\gamma) \, dy\\
    &=:& I_{2a}+I_{2b}.
\end{eqnarray*}
Following similar reasoning as in $R_1$, one obtains
\begin{eqnarray*}
    I_{2a}\leq  C_\gamma (1+|x|^\gamma)(1+t^{\frac{\gamma}{2}}). 
\end{eqnarray*}
For the other integral, we make a change of variable in order to obtain that 
\begin{eqnarray*}
    I_{2b}&=&C\int_{\{t^{\frac{2s-1}{2s}}<|w|\leq t^{-\frac{1}{2s}}\}} \frac{1+|x-t^{\frac{1}{2s}}w|^\gamma}{1+|w|^{N+2s}} \, dw\\
    &\leq& C \int_{\{t^{\frac{2s-1}{2s}}<|w|\leq t^{-\frac{1}{2s}}\}}  \frac{1+C_\gamma(|x|^\gamma+t^\frac{\gamma}{2s}|w|^\gamma)}{1+|w|^{N+2s}} \, dw\\
    &\leq& C_\gamma \left[  (1+|x|^\gamma) \int_{\R^N} \dfrac{1}{1+|w|^{N+2s}}\, dw + t^\frac{\gamma}{2s} \int_{\R^N} \frac{|w|^\gamma}{1+|w|^{N+2s}}\, dw\right] \\
    &\leq &C_\gamma (1+|x|^\gamma)+C_\gamma t^\frac{\gamma}{2s},
\end{eqnarray*}
where we have used that both integrals are convergent since $\gamma<2s$. Consequently,
\begin{eqnarray*}
    I_2\leq C_\gamma (1+|x|^\gamma)(1+t^{\frac{\gamma}{2}}+t^{\frac{\gamma}{2s}}).
\end{eqnarray*}
\item In $R_3:=\{y \in \R^{N}: |x-y|^{2s} \leq t \leq 1\}$ and $R_4:=\{y\in \R^{N}: t \geq 1 \text{ or } |x-y| \geq 1\}$, we follow similar arguments done to estimate $I_1$ and $I_{2b}$ respectively implying that 
\begin{eqnarray*}
    I_3:= \frac{C}{(4\pi t)^{N/2}}\int_{R_3} e^{-\frac{|z|^2}{4t}} (1+|y|^\gamma)\, dy,
\end{eqnarray*}
and
\begin{eqnarray*}
    I_4:=C\int_{R_4} \, \frac{1}{t^{\frac{N}{2s}}\left(1+|t^{-\frac{1}{2s}}(x-y)|^{N+2s}\right)} (1+|y|^\gamma) \, dy,
\end{eqnarray*}
satisfy the same bounds
\end{itemize}
Thus, we clearly get that
\begin{equation*}
  \dfrac{|(T_t u)(x)|}{(1+|x|^\gamma)} \leq \dfrac{C}{(1+|x|^\gamma)}(I_1 + I_2 + I_3 + I_4) \leq C_\gamma(1+t^{\frac{\gamma}{2}}+t^{\frac{\gamma}{2s}}),
\end{equation*}
which conclude (\ref{def_1_X_gam_alt}) and shows that $\|T_t\|_{X_\gamma} \leq C_\gamma (1+t^{\frac{\gamma}{2}}+t^{\frac{\gamma}{2s}})$.\\
We verify now (\ref{def_2_X_gam_alt}) for $T_t u$. For that let us fix $\e>0$. By the uniform continuity  (\ref{def_2_X_gam_alt}) of $u$, there exists $\delta:=\delta(\tilde{\e})>0$ with $\tilde{\e}=\frac{\e}{C_\gamma (1+t^{\frac{\gamma}{2}}+t^{\frac{\gamma}{2s}})}$, where $C_\gamma$ is the one given in the estimate   $ \int_{\R^{N}} \He(t,y)(1+|x-y|^\gamma)  \, dy \leq C_\gamma (1+|x|^\gamma)(1+t^{\frac{\gamma}{2}}+t^{\frac{\gamma}{2s}})
$, such that, for $|z|\leq \delta$, we have
\begin{eqnarray*}
|T_tu(x+z)-T_tu(x)|&=&\left| \int_{\R^{N}} \He(t,y)(u(x+z-y)-u(x-y)) \, dy\right|\\
&\leq& \int_{\R^{N}} \tilde{\e}\,\, \He(t,y)(1+|x-y|^\gamma)  \, dy.
\end{eqnarray*}
Then,
 \begin{eqnarray*}
     |T_tu(x+z)-T_tu(x)| \leq \varepsilon (1+|x|^\gamma),
 \end{eqnarray*}
 so, consequently, for all $|z|\leq \delta$, we conclude that
\begin{equation*}
    \dfrac{|T_tu(x+z)-T_tu(x)|}{1+|x|^\gamma} \leq \e,
\end{equation*}
as required. 
\end{proof}
\begin{prop}
 $(T_t)_{t>0}$ is a strongly continuous semigroup in $X_\gamma$. More precisely, we have that $\lim_{t \to 0} \|T_tu-u\|_{X_\gamma}=0$, for all $u\in X_\gamma$. 
    \end{prop}
    \begin{proof}
To verify this, let $x \in \R^{N}$ and $t>0$ close to zero. We analyze the expression
\begin{eqnarray*}
    \frac{|(T_t u)(x) - u(x)|}{1 + |x|^\gamma} 
&\leq& \int_{|y| > A} \He(t,y) \frac{|u(x-y) - u(x)|}{1 + |x|^\gamma}\,dy + \int_{|y| \leq A} \He(t,y)\frac{|u(x-y) - u(x)|}{1 + |x|^\gamma}\,dy\\
&=:& I_1 + I_2,
\end{eqnarray*}
where $A > 0$ will be chosen sufficiently large.

Firstly, we observe that if $|y| > A$, by  (\ref{est_mix_kernel}) , (\ref{upper_est_mix_op}) and (\ref{est_frac_kernel}), together with the growth condition (\ref{def_1_X_gam_alt}), since $\gamma<2s$, we get 
    \begin{eqnarray*}
        I_1&\leq& C\int_{|y|>A} \, \, \dfrac{1}{t^{\frac{N}{2s}}(1+|t^{-\frac{1}{2s}} y|^{N+2s})} \cdot \dfrac{|u(x-y)-u(x)|}{1+|x|^\gamma}\, dy\\
        &=& C\int_{|z|>t^{-\frac{1}{2s}}A} \dfrac{1}{1+|z|^{N+2s}} \cdot \dfrac{|u(x-t^{\frac{1}{2s}}z)-u(x)|}{1+|x|^\gamma}\, dz\\
        &\leq&C \int_{|z|>t^{-\frac{1}{2s}}A} \dfrac{1+2|x|^\gamma+t^{\frac{\gamma}{2s}}|z|^\gamma}{(1+|z|^{N+2s})(1+|x|^\gamma)} \, dz\\
        &\leq&C \int_{|z |>t^{-\frac{1}{2s}}A} \dfrac{1}{1+|z|^{N+2s}}\, dz +C \int_{|z|>t^{-\frac{1}{2s}}A} \dfrac{t^{\frac{\gamma}{2s}}|z|^\gamma}{1+|z|^{N+2s}}\, dz\\
        &\leq& C \dfrac{t}{A^{2s}} +C \dfrac{t}{ A^{2s-\gamma}}.
    \end{eqnarray*}
     By choosing $A$ sufficiently large, we can ensure that $I_1\leq \dfrac{\e}{3}$, for some $\e>0$. \\
    On the other hand, for $|y|\leq A$, we can split 
    \begin{eqnarray*}
    I_2&= &\int_{|y| \leq 1} \He(t,y)\frac{|u(x-y) - u(x)|}{1 + |x|^\gamma}\,dy +\int_{1\leq |y|<A} \He(t,y)\frac{|u(x-y) - u(x)|}{1 + |x|^\gamma}\,dy\\ 
    &=:&I_{2a}+I_{2b}.
    \end{eqnarray*}
    For $I_{2a}$, by using again (\ref{est_mix_kernel}), (\ref{upper_est_mix_op}),  and (\ref{gauss_ker}), we clearly obtain
     \begin{eqnarray*}
         I_{2a}&\leq& C \sum_{i=1}^3 \, \int_{|y|\leq 1}  \He(t,y)\frac{|u(x-y) - u(x)|}{1 + |x|^\gamma}  \, \chi_i(y)\,dy, 
     \end{eqnarray*}
where $\chi_1(y) = \mathds{1}_{\{|y|^2 < t < |y|^{2s} \leq 1\}}(y), \chi_2(y) = \mathds{1}_{\{t < |y|^2 \leq 1\}}(y)$ and $\chi_3(y) = \mathds{1}_{\{|y|^{2s} \leq t \leq 1\}}(y)$. \\  
Thus, for $\e>0$, by taking $t$ small enough such that $\max\left\{2z\sqrt{t}, \, zt^{\frac{1}{2s}}\right\}\leq \delta$, and applying the uniform continuity condition (\ref{def_2_X_gam_alt}), one shows that 
\begin{eqnarray*}
   C\, \sum_{i=1}^3 \, \int_{|y|\leq 1}  \He(t,y)\frac{|u(x-y) - u(x)|}{1 + |x|^\gamma}  \, \chi_i(y)\,dy = \mathcal{A}+\mathcal{B}+\mathcal{C}\leq \dfrac{\e}{3},
\end{eqnarray*}
where
\begin{eqnarray*}
\mathcal{A}:= C \int_{\{|2\sqrt{t}z|^2 < t < |2\sqrt{t}z|^{2s} \leq 1\}}  e^{-|z|^2} \frac{|u(x-2\sqrt{t}z) - u(x)|}{1 + |x|^\gamma} \, dz,
\end{eqnarray*}
\begin{eqnarray*}
    \mathcal{B}:= C \int_{\{t < |2\sqrt{t}z|^2 \leq 1\}} e^{-|z|^2}  \frac{|u(x-2\sqrt{t}z) - u(x)|}{1 + |x|^\gamma}  \,dz +C \int_{\{t < |t^{\frac{1}{2s}}z|^2 \leq 1\}} \dfrac{1} {1+|z|^{N+2s}}\frac{|u(x-t^{\frac{1}{2s}}z) - u(x)|}{1 + |x|^\gamma}  \,dz,
\end{eqnarray*}
and
\begin{eqnarray*}
     \mathcal{C}:= C \int_{\{|2\sqrt{t}z|^2 < t < |2\sqrt{t}z|^{2s} \leq 1\}}  e^{-|z|^2} \frac{|u(x-2\sqrt{t}z) - u(x)|}{1 + |x|^\gamma} \, dz.
\end{eqnarray*}
    For $I_{2b}$, by (\ref{upper_est_mix_op}), choosing $t^{\frac{1}{2s}}z\leq \delta$, applying (\ref{def_2_X_gam_alt}) one more time, it follows
     \begin{eqnarray*}
         I_{2b}&\leq& C\int_{1<|y|<A} \dfrac{1}{t^{\frac{N}{2s}}(1+|t^{-\frac{1}{2s}} y|^{N+2s})} \cdot \dfrac{|u(x-y)-u(x)|}{1+|x|^\gamma}\, dy\\
         &=& C\int_{t^{-\frac{1}{2s}}<|z|<At^{-\frac{1}{2s}}} \dfrac{1}{1+|z|^{N+2s}} \cdot \dfrac{|u(x-t^{\frac{1}{2s}}z)-u(x)|}{1+|x|^\gamma}\, dz \\
         &\leq& \dfrac{\e}{3},
     \end{eqnarray*}
Therefore, combining the estimates for $I_1$ and $I_2$, we conclude that
that $\lim_{t \to 0} \|T_tu-u\|_{X_\gamma}=0$. 
\end{proof}
\subsection{The Infinitesimal Generator and a Maximum Principle.}
Let $(T_t)_{t>0}$ be a strongly continuous semigroup acting on a Banach space $X$. Then, its infinitesimal generator, denoted by $-\La$, is defined as  
\begin{eqnarray}\label{generator_of_L}
    -\La u :=\lim_{t\downarrow 0} \dfrac{T_tu-u}{t}, \quad \text{ for } u\in D(\La)\subset X,
\end{eqnarray}
where
\begin{eqnarray*}
    D(\La):=\{u\in X \text{ the limit in } (\ref{generator_of_L})  \text{ exists in } X \text{ as }t \downarrow 0 \}.
\end{eqnarray*}
We define $D_\gamma(\La)$ as the domain of the generator $\La$ when considered it in the Banach space $X_\gamma$.  For further details regarding the properties of (\ref{generator_of_L}), we refer to the reader to \cite{cabre_influence_2013}.  \\
We set the following.
\begin{prop} \label{P_max}
    Let $u\in D_\gamma(\La)$ be such that $u(x)\leq u(x_0)$ for all $x\in \R^{N}$ and for some $x_0\in  \R^{N}$. Then, $\La u(x_0)\geq 0$. 
\end{prop}
\begin{proof}
Since $u(x_0-y)\leq u(x_0)$, and the kernel $\He(t,y)$ is nonnegative, its clear that
\begin{eqnarray*}
    -\La u(x_0)=\lim_{t \to 0} \dfrac{(T_tu-u)(x_0)}{t}=\lim_{t \to 0} \dfrac{\int_{\R^{N}}\He(t,y)(u(x_0-y)-u(x_0))\, dy}{t}\leq  0, \, \text{ for all } t>0.
\end{eqnarray*}
which proves the claim.
\end{proof}

The maximum principle stated above is a fundamental tool in our analysis and is essential for proving our main results.\\  
As a direct consequence, we now establish a comparison principle, which follows from Proposition \ref{P_max} and is crucial in our arguments. In particular, this principle guarantees the convergence of solutions to (\ref{eq_mix}) towards $1$, provided we identify a sufficiently large class of initial conditions in the domain of $\La$. This result is similar to that of \cite{cabre_influence_2013}, is stated below, and we include its proof for completeness.  

\begin{prop}\label{P_P.C}
    Let $N\geq1$, $s\in(0,1), 0\leq \gamma <2s$, and consider the kernel $\He$ satisfying the bounds given in Proposition \ref{P1}.\\
    Suppose that $v\in C^1([0,\infty);X_\gamma)$ is such that  $v(t,\cdot)\in D_\gamma(\La)$ for all $t>0$, and let $c$ be a continuous function in $(0,\infty)\times \R^{N}$ that remains bounded in $(0,T)\times \R^{N}$, for all $T>0$. Addtionally, assume that 
    \begin{itemize}
        \item [a)] $v(0,\cdot)\leq 0$ in $\R^{N}$.
        \item [b)] For all $T>0$, we have $\limsup_{|x|\to \infty}v(t,x)\leq 0$ uniformly in $t\in [0,T]$.
        \item [c)] If $(t,x)\in (0,\infty)\times \R^{N}$ and $v(t,x)>0$, then $(v_t+\La v)(t,x)\leq c(t,x) v(t,x)$.
    \end{itemize}
    Then, $v\leq 0$ in all of $(0,\infty)\times \R^{N}$.
\end{prop}
\begin{proof}
We proceed by contradiction. Suppose there exists \( T > 0 \) such that \( v > 0 \) at some point in \( [0,T] \times \R^{N} \). We define the auxiliary function  
\[
w(t,x) := e^{-at} v(t,x),
\]
where \( a > 0 \) will be chosen later. From assumption  $b)$, \( w \) is bounded above and attains a positive maximum at some point \( (t_0, x_0) \in [0,T] \times \R^{N} \), with \( t_0 > 0 \) by assumption  $a)$.  

Since \( w \in C^1([0,\infty); X_\gamma) \), it satisfies \( w_t(t_0, x_0) \geq 0 \) and, by Proposition \ref{P_max}, \( \La w(t_0, x_0) \geq 0 \). Moreover, using assumption $c)$ and noting that \( v(t_0, x_0) > 0 \), we obtain  
\[
0\leq (w_t + \La w)(t_0, x_0) \leq e^{-a t_0} (c(t_0, x_0) - a) v(t_0, x_0)<0.
\]
Choosing \( a > \|c\|_{L^\infty((0,T)\times\R^{N})} \) ensures that \( c - a < 0 \), leading to a contradiction.  
Thus, \( v \leq 0 \) in \( (0,\infty) \times \R^{N} \), completing the proof.  
\end{proof}

We will apply the previous result in the context described by the following lemma. In this setting, we consider the function ${r}(t)=a e^{\nu t}$.  

\begin{lem}\label{L_P.C}
    Let $N\geq1$, $s\in(0,1), 0\leq \gamma <2s$, and let $\He$ be a kernel satisfying the bounds established in Proposition \ref{P1}.\\
    Suppose that $v\in C^1([0,\infty);X_\gamma)$ satisfies $v(t,\cdot)\in D_\gamma(\La)$ for all $t>0$, and let $c$ be a continuous function in $(0,\infty)\times \R^{N}$ which remains bounded in $(0,T)\times \R^{N}$, for all $T>0$.\\
    Let $r:[0,+\infty)\rightarrow[0,+\infty)$ be a continuous function and define 
    \begin{eqnarray*}
        \Omega_r=\{(t,x)\in (0,\infty)\times \R^{N}: \, |x|<r(t)\}.
    \end{eqnarray*}
    Assume the following conditions hold 
    \begin{itemize}
        \item [a)]$v(0,\cdot)\leq 0$ in $\R^{N}$.
        \item [b)]$v\leq 0$ outsise $\Omega_r$, i.e, in $((0,\infty)\times \R^{N})\setminus \Omega_r$. 
        \item [c)] $v_t+\La v\leq c(t,x)v$ in $\Omega_r$.
    \end{itemize}
    Then, $v\leq 0$ in all of $(0,\infty)\times \R^{N}$.
\end{lem}
\begin{proof}
The proof follows by applying Proposition \ref{P_P.C}. In fact $a)$ and $c)$ are trivially true and to check that $b)$ is true let us fix $T>0$. Then, for any $t\in [0,T]$, if $|x|\geq r(t)$, it follows that $(t,x)\notin \Omega_r$, and consequently $v(t,x)\leq 0$.
    Additionally, since $r$ is continuous on $[0,T]$, it attains a maximum value $M_T$ given by
    \begin{eqnarray*}
        M_T:=\max_{ [0,T]} r \geq r(t), \quad \text{ for all } t\in [0,T].
    \end{eqnarray*}
    Thus, for $|x|\geq M_T\geq r(t)\geq 0$, we deduce that $\limsup_{|x|\to \infty}v(t,x)\leq 0$ uniformly in $t\in [0,T]$.
\end{proof}
Although the above comparison principle (Proposition \ref{P_P.C}) applies to smooth solutions in weighted spaces, it does not directly cover classical solutions with strong point properties. Due to the lack of a smooth formulation for sub-solutions and super-solutions in our setting, Proposition \ref{P_P.C} is not directly applicable in the classical sense. For this reason, we also establish a comparison principle for classical solutions under appropriate decay assumptions. This result complements the semigroup-based framework developed in the previous sections and plays a supporting role in the analysis of asymptotic behavior.

We start with the following. 
\begin{defi}[Classical solution]
Let $u \in C_{x,t}^{2,1}\left((0, +\infty)\times \mathbb{R}^N \right)$ be a non-negative function with an admissible growth at infinity
$$
\int_{\mathbb{R}^N} \frac{u(t,x)}{(1 + |x|)^{N+2s}} \, dx < \infty, \quad \text{for all } t > 0.
$$
We say that it is a classical solution of \eqref{eq_mix}, if pointwise,
$$
\left\{\begin{array}{rl}
u_t + \mathcal{L} \, u = f(u) &  \text{ in } (0, +\infty) \times \mathbb{R}^N,   \\
u(0,x) = u_0(x) &   \text{ in }\mathbb{R}^N .
\end{array}\right.
$$

\end{defi}
\begin{ter}[Comparison Principle for classical solutions]\label{T_C.P}
Let $u$ and $v$ be classical solutions of \eqref{eq_mix} such that
\begin{eqnarray*}
    \partial_t u+\La\, u \leq f(u) \quad \text{ and } \quad
\partial_t v +\La \, v \geq f(v),
\end{eqnarray*}
 where $f$ is Lipschitz continuous. If $u(0,x) \leq v(0,x)$ for all $x \in \mathbb{R}^N$ and for all $t > 0$,
\begin{equation}\label{O_error}
u(t,x) = O(|x|^{-(N+2s)}) \quad \text{and} \quad v(t,x) = O(|x|^{-(N+2s)}) \quad \text{as } |x| \rightarrow +\infty,
\end{equation}
then
\begin{equation*}
u(t,x) \leq v(t,x) \quad \text{for all } (x,t) \in (0,+\infty)\times\mathbb{R}^N .
\end{equation*}
\end{ter}
\begin{proof}
    We define the difference function $w := u - v$. Then, $w$ satisfies $w(0,x)\leq 0$ and in $(0,+\infty)\times \R^N$, we obtain
\begin{equation}\label{ineq_cp}
\partial_t w+\La \,w =\partial_t w - \Delta w + (-\Delta)^s w \leq f(u) - f(v)\leq c(t,x) w, 
\end{equation}
where $|c(t,x)| \leq L$, with $L$ being the Lipschitz constant of $f$.\\
Let $T>0$ fixed by arbitrary. By assumption, the function $w$ belongs to $C_{x,t}^{2,1}\left((0, +\infty)\times \mathbb{R}^N \right)$ and 
by (\ref{O_error}), for all $t\in [0,T]$, we also have 
\begin{eqnarray}\label{O_error_w}
    w(t,x)=O(|x|^{-(N+2s)}) \quad \text{ as } |x| \to +\infty.
\end{eqnarray}
We will show that $w \leq 0$ by proving that its positive part $w^+ = \max(w,0)$ is identically zero. Since $w\in C_{x,t}^{2,1}\left((0, +\infty)\times \mathbb{R}^N \right)$, we have 
\begin{eqnarray*}
    \int_{\R^N} w^+ (-\Delta+(-\Delta)^s) w \, dx \geq0.
\end{eqnarray*}
Thus, using  (\ref{O_error_w}), we can multiply each term of (\ref{ineq_cp}) by $w^+$ and integrate over $[0,T]\times \R^N$ to get
\begin{eqnarray}\label{ineq_of_ip}
    \int_0^T\int_{\mathbb{R}^N} w^+ \partial_t w \, dxdt + \int_0^T\int_{\mathbb{R}^N} w^+ (-\Delta+(-\Delta)^s) w \, dxdt \leq \int_0^T\int_{\mathbb{R}^N} c(x,t)w w^+ \, dxdt.
\end{eqnarray}
Since $(w^+)^2$ and $\partial_t[(w^+)^2]$ are continuous in $(0,T)\times \R^N$, and since $w$ satisfies (\ref{O_error_w}), we also have 
\begin{eqnarray*}
    \dfrac{d}{dt} \left[\int_{\R^N} (w^+)^2 \,dx\right]=\int_{\R^N} \partial_t \, \left[(w^+)^2\right]\, dx.
\end{eqnarray*}
Moreover, using (\ref{O_error_w}), the above facts, we obtain from (\ref{ineq_of_ip}) that 
\begin{equation*}
\frac{1}{2}\int_{\mathbb{R}^N} (w^+(T,x))^2 \, dx \leq L\int_0^T\int_{\mathbb{R}^N} (w^+(t,x))^2 \, dxdt.
\end{equation*}
Now, we define $E(t) := \int_{\mathbb{R}^N} (w^+(t,x))^2 \, dx$, for $t\in [0,T]$. Then,
\begin{equation*}
\frac{1}{2}E(T) \leq L\int_0^T E(t) \, dt,
\end{equation*}
that by Gronwall's inequality, implies $E(T) \leq E(0)e^{2LT} = 0$,
Since $T>0$ is arbitrary, we conclude that $w^+(t,x) = 0$ for all $(t,x) \in (0,+\infty)\times\mathbb{R}^N $, which implies $w(t,x) \leq 0$, or equivalently, $u \leq v$ in $(0,+\infty)\times\mathbb{R}^N $.

\end{proof}
\begin{obs}\label{Obs_classic}
Note that if $u_0$ belongs to $D_0(\La)$ then the mild solution of (\ref{eq_mix}) is a global in time classical solution, see for instance (\cite{cabre_influence_2013}, Remark 2.6).
The assumptions on $u$ and $v$ done in (\ref{O_error}) may seem restrictive and not necessarily optimal. However, they provide a sufficient framework for proving our result. In fact, in Theorem \ref{Main_Thconvergence}, we prove the propagation of solutions of (\ref{eq_mix}) with this type of decay at infinity.
\end{obs}

\section{Bounds on the semigroup.}
The technical lemmas in this section are similar to those presented in \cite{cabre_influence_2013}. For the sake of completeness and clarity, we include enough detailed of some to ensure proofs that our exposition is entirely self-contained.\\
From now on we write the notion of mild solution of (\ref{eq_mix}) as 
\begin{eqnarray*}
 u( t,x)&=&(T_t u_0)(x)+\int_{0}^{t}T_{t-s} f(u(r,x))\,dr\\
 &=&\int_{\R^{N}} \He( t,x-y) u_{0}(y)\, d y+\int_{0}^{t} \int_{\R^{N}} \He(t-s,x-y) f(u(r,y))\, d y d r, 
\end{eqnarray*}
for $t\in (0,T)$ and $x\in \R^{N}$, where as we know $u_0\in L^1(\R^{N}) \cap L^\infty (\R^{N})$ is given. We present now an auxiliary lemma that will help us to handle the $C_0(\R^{N})$ initial data in the Theorem \ref{Main_Thconvergence}. 

\begin{lem}\label{L2}
Let $N\geq1$, $ s \in(0,1)$ and $\He$ be a kernel that satisfies the bounds of Proposition \ref{P2}. Then, for some positive constants $c, C$ depending only on $ s$, we have:
\begin{equation}\label{upper_bound_T_t}
    T_t v_0(x) \leq C\left(1+r_0^{-2 s} t\right) a_0|x|^{-N-2 s}, \quad \text { for all } t\geq 1, x \in \R^{N},
\end{equation}
and
$$
\begin{aligned}
T_t v_0(x) \geq c \, \dfrac{t}{t^{\frac{N}{2s}+1}+1} a_0 |x|^{-N-2s}, \quad \text { if } t\geq 1,|x| \geq r_0,
\end{aligned}
$$ 
where
\begin{eqnarray*}
  v_0(x)=  \left\{
\begin{array}{c}
a_0|x|^{-N-2 s} \quad \text { for }  \quad |x| \geq r_0,  \\ 
a_0 r_0^{-N-2 s} \quad  \text { for } \quad  |x| \leq r_0,
\end{array}
\right. \quad a_0>0, r_0\geq 1.
\end{eqnarray*}

\end{lem}
\begin{proof}
We provide a brief outline of the argument, following an approach similar to that in (\cite{cabre_influence_2013}, Lemma 2.15).
For the upper bound, since $t\geq 1$, by (\ref{est_frac_kernel}) and (\ref{est_mix_kernel}), we get
\begin{eqnarray*}
    T_tv_0(x) &=& \int_{\R^{N}} \He(t,x-y)v_0(y)dy \\
    &\leq& B\int_{\mathbb{R}^N} \frac{t^{-N/(2s)}}{1 + |t^{-1/(2s)}(x-y)|^{N+2s}}v_0(y)dy\\
    &\leq& B\left(\int_{B_{|x|/2}(0)}\frac{t^{-N/(2s)}}{1 + |t^{-1/(2s)}(x-y)|^{N+2s}}v_0(y)dy + \int_{\mathbb{R}^N\setminus B_{|x|/2}(0)}\frac{t^{-N/(2s)}}{1 + |t^{-1/(2s)}(x-y)|^{N+2s}}v_0(y)dy\right) \\
    &=:& I_1 + I_2.
\end{eqnarray*}
Following the similar ideas given in (\cite{cabre_influence_2013}, Lemma 2.15), we can deduce that
\begin{eqnarray*}
    I_1 \leq  C \dfrac{\, t}{|x|^{N+2s}}a_0r_0^{-2s},\quad \text{ and } \quad I_2 \leq Ca_0|x|^{-N-2s}.
\end{eqnarray*}
Combining the bounds for $I_1$ and $I_2$, we conclude (\ref{upper_bound_T_t}).

To obtain the lower bound we assume $|x| \geq r_0 \geq 1$. We first observe that 
\begin{eqnarray*}
     T_t v_0(x) &=& \int_{\mathbb{R}^N} \He(x-y, t) v
     _{0}(y)\, d y \geq  \int_{B_1(x)}  \He(x-y, t) v_0(y) d y.
\end{eqnarray*}
Then, using that \(|x-y|^{2s}\leq 1\leq t\) and applying (\ref{est_mix_kernel}) together with (\ref{est_frac_kernel}), the same arguments as in (\cite{cabre_influence_2013}, Lemma 2.15), we obtain that
\begin{eqnarray*}
    T_t v_0(x) \ge c\,\frac{t}{t^{\frac{N}{2s}+1}+1}a_0|x|^{-N-2s},
\end{eqnarray*}
as wanted.
\end{proof}
Before proceeding, we establish the following estimate, which provides upper and lower bounds for the action of \(T_t\) on the function \(w_\gamma(x) = |x|^\gamma\). These bounds will be crucial in our analysis of the long-term behavior of solutions.
\begin{lem}\label{Lbound_w_gam}
    Let $N\geq 1$ $s\in (0,1), \, \gamma\in (0,2s)$ and $\He$ be a kernel that satisfies the bounds of Proposition \ref{P2}. Let $w_\gamma (x):=|x|^\gamma$. Then, for some positive constants $c_\gamma$ and $C_\gamma$, we have 
\begin{eqnarray*}
    T_tw_\gamma(x) \leq C_\gamma\left(|x|^\gamma + t^{\gamma/2s}\right)\quad
    \text{ for all } t \geq 1, x \in \mathbb{R}^N,
\end{eqnarray*}
and
\begin{eqnarray*}
    T_tw_\gamma(x) \geq c_\gamma|x|^\gamma \quad \text{ if } t \geq 1, |x| \geq t^{1/2s}.
\end{eqnarray*}
\end{lem}
\begin{proof}
Since \(t \ge 1\), the kernel \(\He\) satisfies the estimates given in (\ref{est_mix_kernel}) and (\ref{est_frac_kernel}). The details of the proof follow the lines of (\cite{cabre_influence_2013}, Lemma 2.15), so we omit them here. In essence, one can show by scaling arguments that the main contribution to \(T_t w_\gamma(x)\) leads to the upper bound
\[
T_t w_\gamma(x) \le C_\gamma\Bigl(|x|^\gamma + t^{\gamma/(2s)}\Bigr), \quad \text{for all } x\in \R^N.
\]
Moreover, by restricting the integration to the region \(\{|x|\ge t^{1/(2s)}\}\) and using the corresponding lower estimates for \(\He\), one obtains
\[
T_t w_\gamma(x) \ge c_\gamma\,|x|^\gamma.
\]

\end{proof}
The following lemma will be the key one in proving the non-existence of traveling waves, establishing invasion rates, and demonstrating convergence to $1$ behind the front.
\begin{lem}\label{L3}
    Let $N\geq 1$, $s \in (0,1)$, $f$ satisfy (\ref{hyp_f}), and $\He$ be a kernel satisfying the properties of Proposition \ref{P1}. Then, for every $0 < \sigma < \frac{f'(0)}{N+2s}$, there exist $t_0 \geq 1$ and $0 < \varepsilon_0 < 1$ depending only on  $s$, $f$, and $\sigma$, for which the following holds.\\
Given $r_0 \geq 1$ and $0 < \varepsilon \leq \varepsilon_0$, then, the mild solution $v$ of (\ref{eq_mix}) with initial condition 
\begin{eqnarray*}
  v_0(x)=  \left\{
\begin{array}{c}
\quad a_0|x|^{-N-2 s}, \quad \text { for }  \quad |x| \geq r_0,  \\ 
\e=a_0 r_0^{-N-2 s}, \quad  \text {for } \quad  |x| \leq r_0,
\end{array}
\right. \quad a_0>0,
\end{eqnarray*}
 satisfies 
\begin{eqnarray*}
    v(kt_0, x) \geq \varepsilon \quad \text{ for } \, |x| \leq r_0 e^{\sigma kt_0} \text{ and } k \in \mathbb{N} \cup \{0\}.
\end{eqnarray*}
\end{lem}
\begin{proof}
The proof is based on an iterative argument similar to that in (\cite{cabre_influence_2013}, Lemma 3.1). When $t\geq 1$, the fractional part of the kernel is dominant, and one can show that the solution preserves its profile over each time interval of length $t_0$. 
More precisely, starting from the initial condition $v_0$, by applying the comparison principle given in Proposition \ref{P_P.C} one can prove that after time $t_0$ the solution $v(t_0,\cdot)$ satisfies
\begin{eqnarray*}
    v(t_0,x) \geq v_1(x):=  
\begin{cases}
a_1 |x|^{-N-2s} & \text{for } |x| \geq r_1, \\
\varepsilon = a_1 r_1^{-N-2s} & \text{for } |x| \leq r_1,
\end{cases}
\end{eqnarray*}
where $r_1\geq r_0e^{\sigma t_0}$ and $a_1=\e r_1^{N+2s}$. 
\tikzset{every picture/.style={line width=0.75pt}}
\begin{center}
\begin{tikzpicture}[x=0.4pt,y=0.4pt,yscale=-1,xscale=1]

\draw [color={rgb, 255:red, 0; green, 0; blue, 0 }  ,draw opacity=1 ][line width=1.5]  (135,302.5) -- (941,302.5)(551,36.5) -- (551,333.5) (934,297.5) -- (941,302.5) -- (934,307.5) (546,43.5) -- (551,36.5) -- (556,43.5)  ;
\draw [color={rgb, 255:red, 0; green, 0; blue, 0 }  ,draw opacity=1 ][line width=1.5]    (450,201.5) -- (652,200.5) ;
\draw  [dash pattern={on 0.84pt off 2.51pt}]  (450,201.5) -- (451,307) ;
\draw  [dash pattern={on 0.84pt off 2.51pt}]  (652,200.5) -- (650,304.5) ;
\draw [color={rgb, 255:red, 3; green, 46; blue, 254 }  ,draw opacity=1 ]   (351,200.5) -- (750,199.5) ;
\draw  [dash pattern={on 0.84pt off 2.51pt}]  (351,200.5) -- (352,306) ;
\draw  [dash pattern={on 0.84pt off 2.51pt}]  (750,199.5) -- (751,305) ;
\draw [color={rgb, 255:red, 0; green, 42; blue, 255 }  ,draw opacity=1 ]   (137,272.5) .. controls (210,266.5) and (320,257.5) .. (351,200.5) ;
\draw [color={rgb, 255:red, 0; green, 0; blue, 0 }  ,draw opacity=1 ]   (134,284.5) .. controls (159,282.5) and (128.13,287.01) .. (183,282.5) .. controls (237.87,277.99) and (385.31,246.92) .. (405.89,237.16) .. controls (426.47,227.4) and (442.3,215.65) .. (450,201.5) ;
\draw [color={rgb, 255:red, 0; green, 0; blue, 0 }  ,draw opacity=1 ]   (652,200.5) .. controls (655,220.5) and (741,294.5) .. (939,294.5) ;
\draw [color={rgb, 255:red, 0; green, 0; blue, 0 }  ,draw opacity=1 ]   (750,199.5) .. controls (752.39,215.44) and (829,295.5) .. (940,285.5) ;

\draw (437,317.4) node [anchor=north west][inner sep=0.75pt]    {$-r_{0}$};
\draw (647,313.4) node [anchor=north west][inner sep=0.75pt]    {$r_{0}$};
\draw (582,170.4) node [anchor=north west][inner sep=0.75pt]    {$\varepsilon $};
\draw (765,253.4) node [anchor=north west][inner sep=0.75pt]    {$v_{0}( x)$};
\draw (336,318.4) node [anchor=north west][inner sep=0.75pt]  [color={rgb, 255:red, 0; green, 72; blue, 255 }  ,opacity=1 ]  {$-r_{1}$};
\draw (746,313.4) node [anchor=north west][inner sep=0.75pt]  [color={rgb, 255:red, 7; green, 108; blue, 222 }  ,opacity=1 ]  {$r_{1}$};
\draw (796,202.4) node [anchor=north west][inner sep=0.75pt]  [color={rgb, 255:red, 1; green, 71; blue, 246 }  ,opacity=1 ]  {$v_{1}( x)$};

\draw   (551, 201) circle [x radius= 5, y radius= 5]   ;
\draw   (450.96, 302.5) circle [x radius= 5, y radius= 5]   ;
\draw   (650.04, 302.5) circle [x radius= 5, y radius= 5]   ;
\draw   (551, 200) circle [x radius= 5, y radius= 5]   ;
\draw   (351.97, 302.5) circle [x radius= 5, y radius= 5]   ;
\draw   (750.98, 302.5) circle [x radius= 5, y radius= 5]   ;
\end{tikzpicture}
   \end{center}
That is the lower bound $v(t_0,\cdot)$ has the same form as the initial condition $v_0$ but with $(r_0,a_0)$ replaced by $(r_1,a_1)$. We can take $v_1$ as the new initial condition and repeat the argument at $[t_0,\, 2t_0]$, obtaining
\begin{eqnarray*}
    v(2t_0,x) \geq
\begin{cases}
a_2 |x|^{-N-2s} & \text{for } |x| \geq r_2, \\
\varepsilon = a_2 r_2^{-N-2s} & \text{for } |x| \leq r_2,
\end{cases}
\end{eqnarray*}
where $r_2\geq r_1 e^{\sigma t_0}\geq r_0 e^{2 \sigma t_0}$. Therefore, we can repeat the argument above successively, now with initial times $t_0, 2 t_0, 3 t_0, \ldots$ and radius $r_1, r_2, r_3, \ldots$, and obtain
$$
v\left(k t_0, x\right) \geq \varepsilon \quad \text { for }|x| \leq r_k, \, \, \text{ where } r_k:=r_0 e^{k\sigma t_0}.
$$
for all $k \in \mathbb{N}\cup \{0\}$, and the statement of the lemma follows.  
\end{proof}


\begin{cor}\label{C1}
Let $N\geq1$, $ s \in(0,1), f$ satisfy (\ref{hyp_f}), $\He$ be a kernel satisfying the properties of Proposition \ref{P1} and $0<\sigma<\frac{f^{\prime}(0)}{N+2 s}$. Let $t_0 \geq 1$ be the time given by Lemma \ref{L3}.
Then, for every measurable initial datum $u_0$ with $0 \leq u_0 \leq 1$ and $u_0 \not \equiv 0$, there exist $\varepsilon \in(0,1)$ and $b>0$ (both depending on $u_0$ ) such that
\begin{equation*}
    u(t, x) \geq \varepsilon \quad \text { for all } t \geq t_0 \text { and }|x| \leq b e^{\sigma t},
\end{equation*}
where $u$ is the mild solution of $(\ref{eq_mix})$ with $u(0, \cdot)=u_0$.
\end{cor}
\begin{proof}
We adapt the iterative propagation strategy from (\cite{cabre_influence_2013}, Corollary 3.2), so here the difference is to obtain bound for a lower barrier such that it is similar to the one in this corollary, such a barrier is defined by  $\underline{v}(t,\cdot):=T_{t} \left(\eta\chi_{B_1(0)}\right)(\cdot)$, where \(\eta>0\) is a fixed constant and \(\chi_{B_1(0)}\) is the indicator function of the unit ball centered at the origin. The function \(\underline{v}\) serves as a sub-solution for the homogeneous problem and provides a lower bound for \(u\). \\
For this, it suffices to take $x\in \R^N$ with $|x|\geq r_0:= t_0^{\frac{1}{2s}}+1$, thus we have that when $|y|<1<|x|$, 
\begin{eqnarray*}
    1\leq t_0^{\frac{1}{2s}}\leq|x|-1\leq|x|-|y|\leq|y-x|\leq |x|+|y|\leq 2|x|
\end{eqnarray*}
Thus, for $t\in [t_0/2, 3t_0/2]$ and $|y-x|\geq 1$ we use the Proposition \ref{P2} to obtain that
\begin{eqnarray*}
    \underline{v}(t,x)
    &\geq& \eta\, C^{-1} \int_{B_1(0)} \dfrac{t^{-\frac{N}{2s}}}{1+(t^{-\frac{1}{2s}}|x-y|)^{N+2s}} \,dy\\
    &=& \eta\,  C^{-1} \int_{B_1(0)} \dfrac{t}{t^{\frac{N}{2s}+1}+|x-y|^{N+2s}} \, dy\\
    &\geq & \eta\,  C^{-1} C(s) \, |x|^{-N-2s},
\end{eqnarray*}
where in last inequality we have used the fact that 
\begin{eqnarray*}
    t^{\frac{N}{2s}+1}+|x-y|^{N+2s} \leq \left(\frac{3t_0}{2}\right)^{\frac{N}{2s}+1} +(2|x|)^{N+2s}\leq C(s) |x|^{N+2s}.
\end{eqnarray*}
Thus, we obtain that 
\begin{eqnarray}\label{ineq_v_lower bound}
      u\left(\dfrac{t_0}{2}+t,x\right) \geq \underline{v}(t,x)\geq a_0 |x|^{-N-2s} \quad \text{ for } t\in [t_0/2, 3t_0/2] \text{ and } |x|\geq r_0:=t_0^{\frac{1}{2s}}+1.
\end{eqnarray}
Therefore, using the Lemma \ref{L3} and follow the approach given in (\cite{cabre_influence_2013}, Corollary 3.2), we can conclude that $ u(t, x) \geq \varepsilon, \text{ if } t \geq t_0 \text{ and } |x| \leq r_0e^{-\sigma2t_0}e^{\sigma t}$,
by taking $t = \tau_0 + kt_0$. This proves the corollary with 
$b = r_0e^{-\sigma2t_0}$.\\
\end{proof}
The following result is an important one to prove the nonexistence of traveling waves.

\begin{lem}\label{L4}
 Let $N\geq 1$, $ s \in(0,1), f$ satisfy (\ref{hyp_f}), and let $\He$ be a kernel satisfying the properties of Proposition \ref{P1}. Let $\sigma_*=\frac{f'(0)}{N+2 s}$. Let $u$ be a solution of (\ref{eq_mix}), where $u_0 \not \equiv 0,0 \leq u_0 \leq 1$ is measurable, and
$u_0(x) \leq C|x|^{-N-2 s} \quad$ for all $x \in \mathbb{R}^N$ and for some constant $C$. Then, for every $\sigma<\sigma_*$ there exists $\varepsilon \in(0,1)$ and $\underline{t}>0$ such that
\begin{eqnarray}\label{ineq_L4}
    u(t, x) \geq \varepsilon \quad \text { for all } t \geq \underline{t} \text { and }|x| \leq e^{\sigma t} \text {. }
\end{eqnarray}
\end{lem}
\begin{proof}
This proof follows the same approach as in (\cite{cabre_influence_2013}, Lemma 1.3). In essence, one use the Corollary \ref{C1} with a parameter $\sigma'$ satisfying $\sigma<\sigma'<\sigma_*$. As a result, there exist constants \(\varepsilon > 0\) and \(b > 0\) (where $b$ is the constant from Corollary \ref{C1} ) such that  
\[
u(t, x) \geq \varepsilon \quad \text{for all } t \geq t_0 \text{ and } |x| \leq b e^{\sigma' t}.
\]
Moreover, note that $b e^{\sigma't}\geq e^{\sigma t}$ for $t$ large enough. Thus, the desired inequality holds for all \( |x| \leq e^{\sigma t} \), proving Lemma \ref{L4}.
\end{proof}

Even if this lemma concerns initial data decaying at infinity, from it we can easily deduce the nonexistence of planar traveling waves (under no assumption of their behavior at infinity, as in the following statement).


\begin{proof}[Proof of Theorem \ref{prop:no_traveling_waves}.]
The proof becomes the same as in (\cite{cabre_influence_2013}, Proposition 1.4), where by contradiction it is assumed that there exists a non-constant traveling wave solution $u(t,x) = \varphi(x + ct)$ with $0 \leq \varphi \leq 1$ and $\varphi \not\equiv 0$. One comes to show that this leads to $\varphi \equiv 1$, a contradiction. In fact, using Lemma \ref{L4} and exploit the concavity of $f$, we can deduce that
\begin{eqnarray*}
    u(t,x) =\varphi(x+ct)= 1,\quad \text{ for all } (t,x) \in (0,\infty) \times \mathbb{R}^N .
\end{eqnarray*}  
\end{proof}
We can finally obtain. 

\begin{proof}[Proof of Theorem \ref{Main_Thconvergence}.]
We briefly describe the main ideas, referring the reader to (\cite{cabre_influence_2013}, Theorem 1.2) for full details.\\
For part $a)$, one uses an upper bound for the semigroup \(T_t\) (see Lemma \ref{L2}) applied to a linearized problem. This yields that, when $\sigma>\sigma_*$, the solution decays uniformly to zero in the region \(\{|x|\ge e^{\sigma t}\}\).

\noindent For part $b)$, one employs a lower bound argument. In essence, given $0<\sigma<\sigma^* $, take $\sigma'\in (\sigma, \sigma_*)$, and apply Corollary \ref{C1} with $\sigma$ replaced by $\sigma'$. Thus, we obtain 
\begin{eqnarray*}
    -u \leq -\e \text{ in } \Omega_r:=\{ (t,x) \in \R^+\times \R \text{ such that } t\geq t_0, \, |x|\leq r(t)= be^{\sigma' t}\}, 
\end{eqnarray*}
for some $b>0$. Using this together with a suitable lower bound (see Lemma \ref{Lbound_w_gam}) and a comparison argument (Lemma \ref{L_P.C}), one shows that \(u(t,x)\to 1\) uniformly in \(\{|x|\le e^{\sigma t}\}\).

\end{proof}

\section*{Data Availability Statement}
Data sharing is not applicable to this article as no datasets were generated or analyzed during the current study


\section*{Conflicts of Interest}

Te authors declare no conflicts of interest.
\section*{Acknowledgments}
B. Barrios and A. Quaas were partially supported by the project {\it An\'alisis de Fourier y Ecuaciones no locales en Derivadas Parciales} Grant PID2023-148028NB-I00 founded by MCIN/ AEI/10.13039/501100011033/ FEDER, UE.
A. Quaas was also partially supported by FONDECYT Grant 1231585. B. Pichucho was supported by ANID BECAS/DOCTORADO NACIONAL 21220159. He also acknowledges the financial support provided by the Universidad Técnica Federico Santa María through its Research Stay Program for Doctoral Students, which partially funded his research stay at Universidad de La Laguna (Spain).



\end{document}